\documentclass[reqno,a4paper]{amsart} \usepackage{amsmath} \usepackage{amssymb,amsfonts} \usepackage{amsthm} \usepackage{enumerate} \usepackage[mathscr]{eucal} \usepackage{eqlist}
\usepackage{cite}
\usepackage{hyperref}
\usepackage{supertabular}
  \theoremstyle{plain} \newtheorem{theorem}{Theorem}[section] 
\newtheorem{lemma}{Lemma}[section]  \newtheorem{conj}{Conjecture}[section] \theoremstyle{definition} 
  \theoremstyle{remark} 
\numberwithin{equation}{section}     
   \theoremstyle{theorem}  
  \DeclareMathOperator{\ord}{ord}

\begin{document}

\title[A modular app. to the gen. Ramanujan-Nagell eq.] {A modular approach to the generalized Ramanujan-Nagell equation}

\author[E. K. Mutlu, M. Le and G. Soydan]{Elif K{\i}z{\i}ldere Mutlu, Maohua Le and G\"{o}khan Soydan}

\address{{\bf Elif K{\i}z{\i}ldere Mutlu}\\ Department of Mathematics \\ Bursa Uluda\u{g} University\\
 16059 Bursa, Turkey}
\email{elfkzldre@gmail.com}

\urladdr{http://orcid.org/0000-0002-7651-7001}

\address{{\bf Maohua Le}\\
	Institute of Mathematics, Lingnan Normal College\\
	Zhangjiang, Guangdong, 524048 China}

\email{lemaohua2008@163.com}
\urladdr{http://orcid.org/0000-0002-7502-2496}

\address{{\bf G\"{o}khan Soydan} \\ 	Department of Mathematics \\ 	Bursa Uluda\u{g} University\\ 	16059 Bursa, Turkey} \email{gsoydan@uludag.edu.tr }
\urladdr{http://orcid.org/0000-0002-6321-4132}

\newcommand{\acr}{\newline\indent}

\thanks{}


\begin{abstract}
Let $k$ be a positive integer. In this paper, using the modular approach, we prove that if $k\equiv 0 \pmod{4}$, $30< k<724$ and $2k-1$ is an odd prime power, then under the GRH, the equation $x^2+(2k-1)^y=k^z$ has only one positive integer solution $(x,y,z)=(k-1,1,2)$. The above results solve some difficult cases of Terai's conecture concerning this equation.
\end{abstract}

\maketitle

\section{Introduction}\label{sec:1}
Let $\mathbb{Z}$, $\mathbb{N}$ be sets of all integers and positive integers respectively. Let $d,k$ be fixed coprime positive integers with $\min\{d,k\}>1$. A class of polynomial-exponential Diophantine equations of the form
\begin{equation}\label{eq.1.1}
x^2+d^y=k^z,\,\,x,y,z\in\mathbb{N}
\end{equation}
is usually called the generalized Ramanujan-Nagell equation. It has a long history and rich content (see \cite{LS}). In 2014, N. Terai \cite{T1} discussed the solution of \eqref{eq.1.1} in the case $d=2k-1$. He proposed the following conjecture:
\begin{conj}\label{con}
For any $k$ with $k>1$, the equation
\begin{equation}\label{eq.1.2}
x^2+(2k-1)^y=k^z,\,\,x,y,z\in\mathbb{N}
\end{equation}
has only one solution $(x,y,z)=(k-1,1,2)$.
\end{conj}

The above conjecture has been verified in some special cases (see \cite[Theorem 3.1]{BB}, \cite[Corollary 1.4]{DGX}, \cite[Theorem 1.2]{FT} and \cite[Proposition 3.3]{T1}). Most of solved cases of Conjecture \ref{con} focus on the case $4\nmid k$, and very little is known in the case $4\mid k$. In \cite{DGX}, M.-J. Deng, J. Guo and A.-J. Xu verified Conjecture \ref{con} when the case $k\equiv 3 \pmod{4}$ with $3\le k\le 499$. For the case $4\mid k$, N. Terai \cite{T1} used some classical methods to discuss \eqref{eq.1.2} for $k\le 30$. However, his method does not apply for $k\in\{12,24\}$. Until 2017, M. A. Bennett and N. Billerey \cite{BB} used the modular approach to solve the case $k\in\{12,24\}$. It follows that the case $4\mid k$ and  $2k-1$ is an odd prime power is a very difficult case to Conjecture \ref{con}. In this paper, using the modular approach we prove the following result:
\begin{theorem}\label{maintheo}
If $4\mid k$, $30<k<724$ and $2k-1$ is an odd prime power, then under the GRH, Conjecture \ref{con} is true.
\end{theorem}

\section{Preliminaries}
This section introduces some well known notions and results that will be used to prove the main result.
\subsection{The modular approach}
The most important progress in the field of the Diophantine equations has been with Wiles' proof of Fermat's Last Theorem \cite{TW},\cite{Wiles}. His proof is based on deep results about Galois representations associated to elliptic curves and modular forms. The method of using such results to deal with Diophantine problems, is called the \textit{modular approach}. After Wiles' proof, the original strategy was strengthened and many mathematicians achieved great success in solving other equations that previously seemed hard. As a result of these efforts, the generalized Fermat equation
\begin{equation}\label{eq.2.1}
Ax^p+By^q=Cz^r,\,\,\text{with}\,\,1/p+1/q+1/r<1,
\end{equation} 
where $p,q,r\in\mathbb{Z}_{\ge 2}$, $A,B,C$ are non-zero integers and $x,y,z$ are unknown integers became a new area of interest. Call an integer solution $(x,y,z)$ to such an equation \textit{proper} if $\gcd(x,y,z)=1$. It was proved that equation \eqref{eq.2.1} has finitely many proper solutions by H. Darmon and A. Granville \cite[Theorem 2]{DG}. We call the triple of exponents $(p,q,r)$ as in \eqref{eq.2.1} the signature of equation. In the cases $1/p+1/q+1/r=1$, the proper solutions give rise to rational points on certain curves of genus one. It is easily demonstrated that for each $p,q,r$, there exist values of $A,B,C$ such that the equation has infinitely many proper solutions (see \cite[Section 6]{DG}). There also exist values of $A,B,C$ such that the equation has no proper solutions; though, for any $A,B,C,$ there are number fields which contain infinitely many proper solutions (see \cite[Subsection 5.4]{DG}). In the cases where $1/p+1/q+1/r>1$, the proper solutions correspond to rational points on certain curves of genus zero. For this case, it is easy to show that there are infinitely many proper solutions of every equation $x^p+y^q=z^r$ (see \cite[Section 7]{DG}).

In recent 30 years, several authors considered many cases of the above equations. Two survey papers which were written 
by M.A. Bennett, I. Chen, S. Dahmen and S. Yazdani \cite{BCDY} and M. A. Bennett, P. Mih\v{a}ilescu and S. Siksek \cite{BMS} are good references for the case $ABC=1$.

One can find the details concerning modular approach in \cite[Chapter 15]{Cohen} and \cite{Si}. 
\subsection{Signature $(n,n,2)$}
Here we give recipes for signature $(n,n,2)$ which was firstly described by M.A. Bennett and C. Skinner \cite{BS} (see also \cite{IK}). We denote by $rad(m)$ the radical of $|m|$, i.e. the product of distinct primes dividing $m$, and by $\ord_{p}(m)$ the largest nonnegative integer $k$ such that $p^k$ divides $m$.

We always assume that $n\geq7$ is prime, and $a, b, c, A, B$ and $C$ are nonzero integers with $Aa$, $Bb$ and $Cc$ pairwise coprime, $A$ and $B$ are $n$th-power free, $C$ squarefree satisfying
\begin{equation} \label{eq.2.2}
	Aa^n+Bb^n=Cc^2.  
\end{equation}
We further assume that we are in one of the following situations:

$(i)$ $abABC\equiv 1 \pmod{2}$ and $b\equiv -BC \pmod{4}$.

$(ii)$ $ab\equiv 1 \pmod{2}$ and either $\ord_{2}(B)=1$ or $\ord_{2}(C)=1$.

$(iii)$  $ab\equiv 1 \pmod{2}$, $ord_{2}(B)=2$ and $C\equiv -bB/4 \pmod{4}$.

$(iv)$ $ab\equiv 1\pmod{2}$, $ord_{2}(B)\in\{3,4,5\}$ and $c\equiv C \pmod{4}$.

$(v)$ $\ord_{2}(Bb^n)\geq6$ and $c\equiv C \pmod{4}$.

In case $(v)$, we will consider the curve
\begin{equation}\label{eq.2.3}
	E_{3}(a,b,c): Y^2+XY=X^3+\dfrac{cC-1}{4}X^2+\dfrac{BCb^n}{64}X,
\end{equation}
which is defined over $\mathbb{Q}$. By \cite[Lemma 2.1 ]{BS},
the conductor of the curve $E=E_{3}(a,b,c)$ is given by
\begin{equation}\label{eq.2.4}
	N(E)=2^\alpha \cdot C^2\cdot rad(abAB)
\end{equation}
where
\[
\alpha=\begin{cases}
	-1 & \mbox{if $i=3$}, \, \mbox{case $(v)$} \, \, \mbox{and} \, \, \ord_{2}(Bb^n)=6,  \\
	0 & \mbox{if $i=3$}, \, \mbox{case $(v)$} \, \, \mbox{and} \, \, \ord_{2}(Bb^n)\geq7.  \\
\end{cases}
\]

\section{Proof of Theorem \ref{maintheo}.}
Here and below, we assume that $(x,y,z)$ is a solution of \eqref{eq.1.2} with $(x,y,z)\neq (k-1,1,2)$. Then, by \eqref{eq.1.2}, we can get $z>y$ immediately. Obviously, if we can prove that the solution $(x,y,z)$ does not exist, then Conjecture \ref{con} is true. The following two lemmas are basic properties on the solution $(x,y,z)$.

\begin{lemma}\label{lem.3.1}
Suppose $4 \mid k$. Then $2 \nmid y$.
\end{lemma}
\begin{proof}
If $2 \mid y$ then \eqref{eq.1.2} implies $x^2+1 \equiv 0 \pmod{4}$, impossible.
\end{proof}	

\begin{lemma}\label{lem.3.2}
If $2k-1$ is an odd prime power, then $2\nmid z$.	
\end{lemma}
\begin{proof}
See Lemma 2.6 of \cite{DGX}.
\end{proof}
\begin{lemma}\label{lem.3.3-1}
Let $F(t)=t+a/t$ be a function of the real variable $t$, where $a$ is a constant with $a>1$. Then $F(t)$ is a strictly decreasing function for $1\le t<\sqrt{a}$.
\end{lemma}
\begin{proof}
Since $F'(t)=1-a/t^2<0$ for $1\le t<\sqrt{a}$, where $F'(t)$ is the derivative of $F(t)$, we obtain the lemma immediately.
\end{proof}

\begin{lemma}\label{lem.3.3}
If $k$ is a power of 2 and $4\mid k$, then Conjecture \ref{con} is true.
\end{lemma}
\begin{proof}
Let $(x,y,z)$ be a solution different from $(k-1,1,2)$ and let $k=2^r$, where $r\ge 2$ is a positive integer. Then by \eqref{eq.1.2}, we have
\begin{equation}\label{eq.3.1}
x^2+(2^{r+1}-1)^y=2^{rz}.
\end{equation}
Recall that $x$ is odd, hence $2^{rz/2}+x$ and $2^{rz/2}-x$ are relatively prime.

If $2\mid rz$, then from \eqref{eq.3.1} we get
\begin{equation*}
2^{rz/2}+x=f^y,\,\,2^{rz/2}-x=g^y,\,\,2^{r+1}-1=fg,\,\,f,g\in\mathbb{N},\,\,2\nmid fg,\,\,\gcd(f,g)=1,
\end{equation*}
whence we obtain
\begin{equation}\label{eq.3.2}
2^{rz/2+1}=f^y+g^y=(f+g)\bigg(\dfrac{f^y+g^y}{f+g}\bigg).
\end{equation}
By Lemma \ref{lem.3.1}, we have $2\nmid y$. Therefore, since $2\nmid fg$, $(f^y+g^y)/(f+g)$ is an odd positive integer. It follows from \eqref{eq.3.2} that $(f^y+g^y)/(f+g)=1$, $y=1$. Take $a=2^{r+1}-1$ and $t=g$. By Lemma \ref{lem.3.3-1}, we have
\begin{equation}\label{eq.3.3}
2^{rz/2+1}=f+g=\dfrac{2^{r+1}-1}{g}+g\le (2^{r+1}-1)+1=2^{r+1}.
\end{equation} 
Recall that $z>y=11$. We find from \eqref{eq.3.3} that $f=2^{r+1}-1$, $g=1$, $z=2$ and $(x,y,z)=(2^r-1,1,2)=(k-1,1,2)$, a contradiction.

If $2\nmid rz$, then $2\nmid r$. Since $2\nmid y$, we see from \eqref{eq.3.1} that the equation
\begin{equation}\label{eq.3.4}
X^2+(2^{r+1}-1)Y^2=4\cdot 2^{Z},\,\,X,Y,Z\in\mathbb{Z},\,\,\gcd(X,Y)=1,\,\,Z>0
\end{equation} 
has a solution
\begin{equation}\label{eq.3.5}
(X,Y,Z)=(x,(2^{r+1}-1)^{(y-1)/2},rz-2).
\end{equation}
Since $1^2+(2^{r+1}-1)\cdot 1^2=4\cdot 2^{(r-1)}$, $(1,1,r-1)$ is a solution and clearly the one with smaller value of $z$, hence is what is called a \textit{minimal} solution in \cite{Le}. Then Theorem 2 of \cite{Le} implies that
\begin{equation}\label{eq.3.6}
rz-2=(r-1)t,\,\,t\in\mathbb{N}.
\end{equation}
However, since $2\mid r-1$, we get from \eqref{eq.3.6} that $2\mid rz$, a contradiction. Thus, the lemma is proved.
\end{proof}
\begin{lemma}\label{lem.3.4}
If $4\mid k$, $2k-1$ is an odd prime power and $k$ is a square, then Conjecture \ref{con} is true.
\end{lemma}
\begin{proof}
Since $k$ is a square, $k=\ell^2$, where $\ell$ is a positive integer with $2\mid \ell$. by \eqref{eq.1.2}, we have $(\ell^z)^2-x^2=(\ell^z+x)(\ell^z-x)=(2k-1)^y$. Since $2k-1$ is an odd prime power and $\gcd(\ell^z+x,\ell^z-x)=1$, we have $\ell^z+x=(2k-1)^y$ and $\ell^z-x=1$, whence we get
\begin{equation}\label{eq.3.7}
2\ell^z=(2k-1)^y+1.
\end{equation}
Further, since $(2k-1)+1=2\ell^2$ and $2 \nmid y$ by Lemma \ref{lem.3.1}, we obtain from \eqref{eq.3.7} that
\begin{equation*}
2\ell^z=2\ell^2\left(\dfrac{(2k-1)^y+1}{(2k-1)+1}\right),
\end{equation*}
whence we get
\begin{equation}\label{eq.3.8}
\ell^{z-2}=\dfrac{(2k-1)^y+1}{(2k-1)+1},
\end{equation}
where $((2k-1)^y+1)/((2k-1)+1)$ is a positive integer. Notice that $2\mid \ell$ and $2\nmid ((2k-1)^y+1)/((2k-1)+1).$ We see from \eqref{eq.3.8} that $z=2$ and $(x,y,z)=(k-1,1,2)$. The lemma is proved. 
\end{proof}

For any fixed positive integers $m$ and $n$ with $n>1$, there exist unique positive integers $f$ and $g$ such that 
\begin{equation}\label{eq.3.9}
m=fg^n,\,\,f\,\,\text{is}\,\,n\text{th-power free.}
\end{equation}
The positive integer $f$ is called the $n$th-power free part of $m$, and denoted by $f(m).$ Similarly, $g$ is denoted by $g(m)$. Obviously, by \eqref{eq.3.9}, if $2\mid m$, then we have
\begin{equation}\label{eq.3.10}
\ord_{2}(m)=\ord_{2}(f(m))+n\ord_{2}(g(m)),\,\,0\le \ord_{2}(f(m))<n.
\end{equation}

Here, we consider equation \eqref{eq.1.2} where $y\ge 7$ is prime and $y=3$ or $y=5$, respectively.

\subsection{The case $y\ge 7$ prime}
Let $k$ be a positive integer with $4\mid k$. Suppose that $y\ge 7$ is prime. Then equation \eqref{eq.1.2}
becomes
\begin{equation}\label{eq.3.11}
(-1)\cdot(2k-1)^y+k^z=x^2.
\end{equation}
Then, the ternary equation \eqref{eq.2.1} can be obtained from \eqref{eq.3.11} by the substitution 
\begin{equation}\label{eq.3.12}
A=-1, \, \, a=2k-1, \, \, B=f(k^z),\, \, b=g(k^z), \, \, C=1,\, \, c=(-1)^{(x-1)/2} x, 
\end{equation}
where $f(k^z)$ and $g(k^z)$ are defined as in \eqref{eq.3.9}.
\begin{lemma}\label{lem.3.5}
If $4\mid k$, then $b$, $B$ and $C$ in \eqref{eq.3.12} satisfy the case $(v)$ with $\ord_{2}(Bb^n)>6$.
\end{lemma}
\begin{proof}
Since $2\mid k$, we see from \eqref{eq.1.2} that $2\nmid x$. So we have $C\equiv (-1)^{(x-1)/2}x\equiv 1 \pmod{4}$. In addition, by \eqref{eq.3.9} and \eqref{eq.3.12}, we have $Bb^n=f(k^z)(g(k^z))^n=k^z$, whence we get $\ord_{2}(Bb^n)=\ord_{2}(k^z)=z\ord_{2}(k)>6$ (recall that $z>y$). Thus, the lemma is proved.
\end{proof}

By \eqref{eq.2.3} and \eqref{eq.3.12}, we are interested in the following elliptic curve (called a Frey curve)
\begin{equation}\label{eq.3.13}
E_{3}: Y^2+XY=X^3+\dfrac{(-1)^{(x-1)/2}x-1}{4}X^2+\dfrac{k^z}{64}X.
\end{equation}
By Lemma \ref{lem.3.5} and \eqref{eq.2.4}, the conductor of this elliptic curve is
\begin{equation}\label{eq.3.14}
N(E_3)=rad(2k-1)\cdot rad(k).
\end{equation}
Note that when $k=720$, one gets that $N(E_3)=43170$ and $2k-1=1439$ is prime. But when $k=724$, one obtains $N(E_3)=523814$, outside the range of the Cremona elliptic curve database \cite{Cremona} where the upper bound for conductors is 500000. We therefore restrict attention to $30<k\le 720$. 

Using Lemmas \ref{lem.3.3} and \ref{lem.3.4}, we can exclude the cases $k=2^{r_0}$ with $r_0=6$ and $k=\ell^2$ with $\ell\in\{6,8,10,18,22,24\}$. This leaves 50 values of $k$ to consider. We proceed as follows.

For a given $k$ we compute by \eqref{eq.3.14} the conductor of the Frey curve at \eqref{eq.3.13}. Using Cremona's elliptic curve database \cite{Cremona} we obtain a list of isomorphism classes of elliptic curves for that conductor. In each class, we must determine whether there exists a model consistent with the model at \eqref{eq.3.13}. For example, when $k=192$ and the conductor is 2298, the isomorphism class of the curve labelled 2298.h4, $[1,0,0,-184,1088]$, contains the curve $[1,-48,0,576,0]$ (note that this fails to provide a solution to our problem, because the corresponding values of $y$, $z$, equal 1,2, not allowed). Since the curve \eqref{eq.3.13} has point $(0,0)$ of order 2, it is only necessary to consider isomorphism classes determined by curves with nontrivial 2-torsion. Suppose a Cremona class representative has nontrivial 2-torsion point $T_0$. To obtain an isomorphic curve of the form \eqref{eq.3.13} we must take the transformation mapping $T_0$ to $(0,0)$, and then test the resulting curve to see whether the $X-$ coefficient is of the form $\dfrac{k^z}{64}$. This was programmed into \textsc{Magma}. Resulting curves with corresponding $(y,z)=(1,2)$ are not allowed, and only one other curve resulted, namely $[1, 733/4, 0, 33/16, 0]$ when $k=132$ with $z=2$. But this does not provide a solution to our problem because there is no corresponding value of $x$ (or $y$). Finally, thus, we deduce \eqref{eq.1.2} has no solutions where $y\ge 7$.

\subsection{The case $y=3$ or $y=5$}

Here we solve the Diophantine equations
\begin{equation}\label{eq.3.15}
x^2+(2k-1)^3=k^z,\,\, z>3\,\,\text{odd},
\end{equation}
and
\begin{equation}\label{eq.3.16}
x^2+(2k-1)^5=k^z,\,\, z>5\,\,\text{odd},	
\end{equation}
where $4\mid k$, $30<k <724$ and $2k-1$ is an odd prime power.

Write in \eqref{eq.1.2} $y=6A+i$, $z=3B+j$ where $i=3$ or $5$ and $0\leq j\leq 2$, $A,B\ge 0$. Since $2k-1$ is an odd prime power, we have $2k-1=p^r$, where $p$ is an odd prime and $r$ is a positive integer. Then we see that 
\begin{equation*}
	\left(\frac{k^{B+j}}{(2k-1)^{2A}},\frac{xk^j}{(2k-1)^{3A}}\right)
\end{equation*}
is an $S$-integral point $(U,V)$ on the elliptic curve 
\begin{equation*}
	{\mathcal{E}}_{ijk}: V^2 = U^3 - (2k-1)^ik^{2j} \:,
\end{equation*}
where $S=\{p\}$, $4\mid k$, $30< k<724$ and $2k-1$ is a power of $p$, in view of the restriction $\gcd(k,x)=1$.

A practical method for the explicit computation of all 
$S$-integral points on a Weierstrass elliptic curve has been developed by
A. Peth\H{o}, H.G. Zimmer, J. Gebel and E. Herrmann in \cite{PZGH} and has been
implemented in \textsc{Magma} \cite{Magma}. 
The relevant routine $\mathtt{SIntegralPoints}$
worked without problems for all triples $(i,j,k)$ except for\\ $(i,j,k)\in\{(5,2,96),(5,1,120),(5,2,156),(5,2,180),(5,2,192),(5,2,220),$\\$ (3,1,232), (5,0,232), (5,2,232), (5,0,240),(5,2,240),(5,2,244),(5,0,304),$\\$ (5,1,304), (5,2,304),(3,2,316),(5,0,316), (5,2,316), (5,2,324),(5,0,360),$\\$(5,1,364), (5,2,364),(3,2,372), (5,1,372), (5,2, 372),(5,2,376),(3,1,412),$\\$(3,2,412), (5,0,412),(5,0,420),(5,0,432), (5,1,432),(3,2,444), (5,1,444),$\\$(5,2,444),(5,0,456), (5,1,456),(5,2,460),(5,1,492),(5,1,516), (5,2,516),$\\$(3,1,520),(5,0,520), (5,2,520),(5,2,532),(5,1,544),(5,2,552), (3,2,612),$\\$(5,0,612), (5,1,612),(5,2,612), (5,1,616),(5,0,640),(5,2,640),(3,2,652),$\\$(5,2,652), (5,2,660), (3,2,664) (5,0,664),(5,2,664), (5,1,684),(5,0,700),$\\$(5,1,700), (5,2,700),(5,0,712), (5,0,720),(5,1,720)\}$. The non-exceptional\\ triples $(i,j,k)$ do not give any positive integer solution to equation \eqref{eq.3.15} or \eqref{eq.3.16}.

\subsubsection{\textbf{The elementary approach to some exceptional triples}}
Thirty-eight of the above exceptional triples have been solved using an elementary approach, as follows.
\begin{lemma}\label{lem.3.7}
If $k\equiv 3$ or $4\pmod{5}$, then \eqref{eq.3.16}
has no solutions $(x,z)$.
\end{lemma} 
\begin{proof}
	We now assume that $(x,z)$ is a solution of \eqref{eq.3.16}. If $k\equiv 3 \pmod{5}$, then we have $2k-1\equiv 0\pmod{5}$, and by \eqref{eq.3.16}, $x^2\equiv k^z-(2k-1)^5\equiv 3^z \pmod{5}$. Further, since $2\nmid z$, we get $1=(3^z/5)=(3/5)=-1$, a contradiction, where $(*/*)$ is the Legendre symbol. If $k\equiv 4\pmod{5}$, then we have $x^2\equiv k^z-(2k-1)^5\equiv (-1)^z-2^5\equiv -1-2\equiv 2 \pmod{5}$. Hence, we can get a similar contradiction that $1=(2/5)=-1$. Thus, the lemma is proved.
\end{proof}
\begin{lemma}\label{lem.3.8}
	If $2\mid k$ and $k+1$ has an odd prime divisor $p$ with $p\equiv \pm 3 \pmod{8}$, then \eqref{eq.3.16} has no soutions $(x,z)$.
\end{lemma}
\begin{proof}
	By \eqref{eq.3.16}, since $2\nmid z$, we have $x^2\equiv k^z-(2k-1)^5\equiv (-1)^z-(-3)^5\equiv -1+243\equiv 242\equiv 2\cdot 11^2 \pmod{k+1}$. It implies that, for every odd prime divisor $p$ of $k+1$, we have $1=(2\cdot 11^2/p)=(2/p)=(-1)^{(p^2-1)/8}$ and $p\equiv \pm 1 \pmod{8}$. Therefore, if $p\equiv \pm 3 \pmod{8}$, then \eqref{eq.3.16} has no solutions $(x,z)$.
\end{proof}

Notice that if $2\mid k$ and every odd prime divisor $p$ of $k+1$ satisfies $p\equiv \pm 1 \pmod{8}$, then either $k+1\equiv 1 \pmod{8}$ or $k+1\equiv -1 \pmod{8}$. Hence, by Lemma \ref{lem.3.8}, we can obtain the following lemma immediately.
\begin{lemma}\label{lem.3.9}
	If $k\equiv 2$ or $4\pmod{8}$, then \eqref{eq.3.16} has no solutions $(x,z)$.
\end{lemma}
\begin{lemma}\label{lem.3.10}
For $k\in \{120,156,180,220,244,304,316,324,360,364,372,376,\\412,420,444,460,492,516,532,544,612,652,660,664,684,700\},$ \eqref{eq.3.16} has no solutions $(x,z)$.
\end{lemma}
\begin{proof}
By Lemmas \ref{lem.3.7}, \ref{lem.3.8} and \ref{lem.3.9}, \eqref{eq.3.16} has no solutions $(x,z)$ for $k\in\{244,304,324,364,444,544,664,684\}$, $k\in\{120,360,376\}$ and $k\in\{156,180,220,\\316,372,412,420,460,492,516,532,612,652,660,700\}$, respectively.
\end{proof}

Denote the rank of the elliptic curve ${\mathcal{E}}_{ijk}$ by $r$. Here, we seperate the above remaining twentynine exceptional triples $(i,j,k)$ depending on whether $r=0$, $r=1$ and $r=2$, respectively.
\subsubsection{\textbf{The case $r=0$}}
For the triples $(i,j,k)\in \{(5,1,456),(5,2,552),(5,1,616),\\(3,2,652),(5,1,720)\}$, there are no rational points (so no $S$-integral points) on ${\mathcal{E}}_{ijk}$ under the assumption that $r=0$ which is proved by descent algorithms of \textsc{Magma}.    
\subsubsection{\textbf{The case $r=1$}}
For each remaining triples $(i,j,k)$ (in total twentyfour triples), the rank of ${\mathcal{E}}_{ijk}$ is 1, i.e. $r=1$ except for $(i,j,k)=(3,2,664)$. We performed two-, four- and eight-descent algorithms or two-, four-, three- and twelve-descent algorithms for these triples. Since \textsc{Magma} found a generator for each of them, it was succesfull to show non-existence of $S$-integral points on ${\mathcal{E}}_{ijk}$ for the exceptional twentythree triples.

The rank 1 curves frequently have generators of large height. We can estimate the height of a generator in advance using the G-Z formula \cite{GZ}, as for example used by A. Bremner \cite{Br} in treating the family of curves $y^2=x(x^2+p)$. Having an estimate of the height in advance tells us whether it is likely that standard descent arguments, as programmed into \textsc{Magma}, will be successfull in finding the generator. 
For twentythree curves we are considering here, \textsc{Magma} was able to compute generators for all cases, using a combination of three-, four-, eight-, and twelve-descent algorithms.

\subsubsection{\textbf{The case $r=2$}}
The single instance of rank 2 was at $(i,j,k)=(3,2,664)$. \textsc{Magma} found only $S$-integral point $(6435758912 : 516297057335360 : 1)$ on the corresponding curve under the assumption that $1\le r\le 2$ and its generator $(402234932 : 8067141520865 : 1)$. By eight-descent algorithm, we found two independent points which are generators. So, it is confirmed that $r=2$ and this curve has only $S$-integral point $(6435758912 : 516297057335360 : 1)$. But it does not give any positive integer solution to equation \eqref{eq.3.15}. 

\bigskip

To sum up, the Theorem \ref{maintheo} is proved.

\bigskip

For the computations, we used  iMAC computer with the following characteristics: Processor Intel i5, 2.7 GHz, 16GB RAM, 1600 MHz DDR3. All computations are done by Magma V2.24-5.

The \textsc{Magma} \cite{Magma} files used to carry out the computations in this paper are available at: 
\href{http://gsoydan.home.uludag.edu.tr/images/programs.zip}{http://gsoydan.home.uludag.edu.tr/images/programs.zip}.

\section*{Acknowledgments} 
We are grateful to Professor Andrew Bremner for his generous helps about \textsc{Magma} computations and his useful ideas and would like to thank Professor Benjamin Matschke for his useful idea that shows how to reduce to S-integral points on Mordell equations \eqref{eq.3.15}-\eqref{eq.3.16} and sharing results of computations of his \textsc{Sage} codes with us and to Professor Jennifer Balakrishnan for enabling my connection with him. We also cordially thank an anonymous referee for carefully reading our paper and for giving such constructive comments which substantially helped improving the quality of the paper. G. S. was supported by the Research Fund of Bursa Uludag University under Project No: F-2020/8.


\begin{thebibliography}{0}
\bibitem{BB} {\sc M. Bennett and N. Billerey}, Sums of two S-units via Frey-Hellegouarch curves, {\em Math. Comp.} {\bf 305} (2017), 1375--1401.

\bibitem{BCDY} {\sc M.A. Bennett, I. Chen, S.R. Dahmen, S. Yazdani}, Generalized Fermat 
equations: a miscellany, {\em Int. J. Number Theory} {\bf 11} (2015), 1--28.

\bibitem{BMS} {\sc M.A. Bennett,  P. Mih\v{a}ilescu, S. Siksek}, The generalized Fermat equation, 
Open Problems in Mathematics (J. F. Nash, Jr. and M. Th. Rassias eds), 173-205, Springer, New York, 2016   

\bibitem{BS} {\sc  M. A. Bennett and C. Skinner}, Ternary Diophantine equations via Galois representations and modular forms, {\em Canad.~J.~Math.} {\bf 56} (2004), 23--54.

\bibitem{Magma} {\sc W. Bosma, J.  Cannon, C.  Playoust}, The Magma 
Algebra System I. The user language, {\em J. Symbolic Comput.} {\bf 24} (1997), 235--265 

\bibitem{Br} {\sc A. Bremner}, On the equation $Y^2=X(X^2+p)$, {\em Number theory and applications (Banff, AB, 1988), 3–22, NATO Adv. Sci. Inst. Ser. C: Math. Phys. Sci.}, 265, Kluwer Acad. Publ., Dordrecht, 1989. 

\bibitem{Cohen} {\sc H. Cohen}, Number Theory Vol. II: Analytic and Modern Tools, Springer, 2007.

\bibitem{Cremona} {\sc J. Cremona}, Elliptic Curve Data, http://johncremona.github.io/ecdata/
and   {\it LMFDB - The $L$-functions and Modular Forms Database}, http://www.lmfdb.org/
or {\it Elliptic curves over $\mathbb Q$},  http://www.lmfdb.org/EllipticCurve/Q/ 

\bibitem{DG}{\sc H. Darmon and A. Granville}, On the equations $z^m=F(x,y)$ and $Ax^p+By^q=Cz^r$. {\em Bull. London Math. Soc.} {\bf 27} (1995), 513--543.

\bibitem{DGX} {\sc M.-J. Deng, J. Guo and A.-J. Xu}, A note on the Diophantine equation $x^2+(2c-1)^m=c^n$, {\em Bull. Aust. Math. Soc.} {\bf 98} (2018), 188--195.

\bibitem{FT} {\sc Y. Fujita and N. Terai}, On the generalized Ramanujan-Nagell equation $x^2+(2c-1)^m=c^n$, {\em Acta Math. Hung.} {\bf 162} (2020), 518--526.

\bibitem{GZ} {\sc B. H. Gross and D. B. Zagier}, Heegner points and derivatives of $L$-series, {\em Invent. Math.} {\bf 84} (1986), 225--320.

\bibitem{IK} {\sc W. Ivorra and A. Kraus}, Quelques r\'{e}sultats sur les \'{e}quations $ax^{p}+by^{p}=cz^{2}$, {\em Canad. J. Math.} (2006),{\bf 58}, 115--153.

\bibitem{Le} {\sc M.-H. Le}, Some exponential Diophantine equations I: The equation $D_1x^2-D_2y^2=\lambda k^z$. {\em J. Number Theory} {\bf 55} (1995), 209--221. 

\bibitem{LS} {\sc M. H. Le and G. Soydan}, A brief survey on the generalized Lebesgue-Ramanujan-Nagell equation, {\em Surv. Math. Appl.} {\bf 15} (2020), 473--523.
 
\bibitem{Si} {\sc S. Siksek}, The modular approach to Diophantine 
equations,  {\em Panoramas \& Synth\`{e}ses} {\bf 36} (2012), 151--179.	

\bibitem{PZGH} {\sc A.~Peth\H{o}, H.G.~Zimmer, J.~Gebel, E.~Herrmann},
Computing all $S$-integral points on elliptic curves, 
{\em Math.~Proc.~Camb.~Phil.~Soc.} {\bf 127} (1999), 383-402.

\bibitem{TW} {\sc R. Taylor and A. Wiles}, Ring-theoretic properties of certain Hecke algebras. {\em Ann. of Math.} {\bf 141} (1995), no. 3, 553--572. 

\bibitem{T1} {\sc N. Terai}, A note on the Diophantine equation $x^2+q^m=c^n$, {\em Bull. Aust. Math. Soc.} {\bf 90} (2014), 20--27.


\bibitem{Wiles} {\sc A. Wiles}, Modular elliptic curves and Fermat's Last Theorem,  {\em  Ann.~of~Math.} {\bf 141} (1995), 443--551.


\end{thebibliography}
\end{document}